\documentclass[11pt,reqno]{amsart}
\usepackage{amssymb,latexsym}
\usepackage{graphicx}
\usepackage{mathtools}
\usepackage{color}
\usepackage[T1]{fontenc}
\usepackage{hyperref}
\usepackage{amsmath}
\usepackage{mathdots}
\usepackage{breqn}
\usepackage[toc,page]{appendix}
\usepackage{url}    
\usepackage{breqn}
\usepackage{CJK}
\newcommand{\bburl}[1]{\textcolor{blue}{\url{#1}}}

\usepackage{float}

\calclayout

\makeatletter
\newcommand{\monthyear}[1]{%
  \def\@monthyear{\uppercase{#1}}}
\newcommand{\volnumber}[1]{%
  \def\@volnumber{\uppercase{#1}}}
\AtBeginDocument{%
\def\ps@plain{\ps@empty
  \def\@oddfoot{\@monthyear \hfil \thepage}%
  \def\@evenfoot{\thepage \hfil \@volnumber}}
\def\ps@firstpage{\ps@plain}
\def\ps@headings{\ps@empty
  \def\@evenhead{%
    \setTrue{runhead}%
    \def\thanks{\protect\thanks@warning}%
    \uppercase{\ }\hfil}%
  \def\@oddhead{%
    \setTrue{runhead}%
    \def\thanks{\protect\thanks@warning}%
    \hfill\uppercase{Generalizing Zeckendorf's Theorem to A Non-constant Recurrence}}%
  \let\@mkboth\markboth
  \def\@evenfoot{%
    \thepage \hfil \@volnumber}%
  \def\@oddfoot{%
    \@monthyear \hfil \thepage}%
  }%
\footskip=25pt
\pagestyle{headings}%
}
\makeatother

\theoremstyle{plain}
\numberwithin{equation}{section}
\newtheorem{thm}{Theorem}[section]
\newtheorem{theorem}[thm]{Theorem}
\newtheorem{lemma}[thm]{Lemma}
\newtheorem{corollary}[thm]{Corollary}
\newtheorem{example}[thm]{Example}
\newtheorem{definition}[thm]{Definition}

\newcommand{\ignore}[1]{}

\newcommand{\lp}{\left(}
\newcommand{\rp}{\right)}

\newcommand{\lf}{\lfloor}
\newcommand{\rf}{\rfloor}
\newcommand{\Lf}{\left\lfloor}
\newcommand{\Rf}{\right\rfloor}

\newcommand\be{\begin{eqnarray}}
\newcommand\ee{\end{eqnarray}}
\newcommand\bea{\begin{eqnarray}}
\newcommand\eea{\end{eqnarray}}
\newcommand\ben{\begin{enumerate}}
\newcommand\een{\end{enumerate}}

\newtheorem{conj}[thm]{Conjecture}

\newtheorem{lem}[thm]{Lemma}

\newtheorem{defi}[thm]{Definition}

\begin{document}

\monthyear{August 2020}
\volnumber{Volume, Number}
\setcounter{page}{1}
\title{Generalizing Zeckendorf's Theorem to a Non-constant Recurrence}

\author[E.~Bo\l dyriew]{El\.zbieta~Bo\l dyriew}
\email{\textcolor{blue}{\href{mailto:eboldyriew@colgate.edu}{eboldyriew@colgate.edu}}}
\address{Department of Mathematics, Colgate University, Hamilton, NY 13346}

\author[A.~Cusenza]{Anna~Cusenza}
\email{\textcolor{blue}{\href{ascusenza@g.ucla.edu}{ascusenza@g.ucla.edu}}}
\address{Department of Mathematics, University of California, Los Angeles, Los Angeles, CA, 90095}

\author[L.~Dai]{Linglong~Dai}
\email{\textcolor{blue}{\href{mailto:dail23@georgeschool.org}{dail23@georgeschool.org}}}
\address{George School, Newtown, PA 18940}

\author[P.~Ding]{Pei~Ding}
\email{\textcolor{blue}{\href{mailto:13764986079@163.com}{13764986079@163.com}}}
\address{Shanghai World Foreign Language School, Shanghai, China}

\author[A.~Dunkelberg]{Aidan~Dunkelberg}
\email{\textcolor{blue}{\href{awd4@williams.edu}{awd4@williams.edu}}}
\address{Department of Mathematics and Statistics, Williams College, Williamstown, MA 01267}

\author[J.~Haviland]{John~Haviland}
\email{\textcolor{blue}{\href{mailto:havijw@umich.edu}{havijw@umich.edu}}}
\address{Department of Mathematics, University of Michigan, Ann Arbor, MI 48109}

\author[K.~Huffman]{Kate~Huffman}
\email{\textcolor{blue}{\href{klhuffman@crimson.ua.edu}{klhuffman@crimson.ua.edu}}}
\address{Department of Mathematics,University of Alabama, Tuscaloosa, AL 35401}

\author[D.~Ke]{Dianhui~Ke}
\email{\textcolor{blue}{\href{kdianhui@umich.edu}{kdianhui@umich.edu}}}
\address{Department of Mathematics, University of Michigan, Ann Arbor, MI 48109}

\author[D.~Kleber]{Daniel~Kleber}
\email{\textcolor{blue}{\href{kleberd@carleton.edu}{kleberd@carleton.edu}}}
\address{Department of Mathematics and Statistics, Carleton College, Northfield, MN 55057}

\author[J.~Kuretski]{Jason Kuretski}
\email{\textcolor{blue}{\href{mailto:jkuretski@usf.edu}{jkuretski@usf.edu}}}
\address{Department of Mathematics and Statistics, University of South Florida, Tampa, FL 33620}

\author[J.~Lentfer]{John~Lentfer}
\email{\textcolor{blue}{\href{mailto:jlentfer@hmc.edu}{jlentfer@hmc.edu}}}
\address{Department of Mathematics, Harvey Mudd College, Claremont, CA 91711}

\author[T.~Luo]{Tianhao~Luo}
\email{\textcolor{blue}{\href{mailto:christopher.luo21@sailsburyschool.org}{christopher.luo21@sailsburyschool.org}}}
\address{Salisbury School, Salisbury, CT 06068}

\author[S.~J.~Miller]{Steven~J.~Miller}
\email{\textcolor{blue}{\href{mailto:sjm1@williams.edu}{sjm1@williams.edu}},  \textcolor{blue}{\href{Steven.Miller.MC.96@aya.yale.edu}{Steven.Miller.MC.96@aya.yale.edu}}}
\address{Department of Mathematics and Statistics, Williams College, Williamstown, MA 01267}

\author[C.~Mizgerd]{Clayton~Mizgerd}
\email{\textcolor{blue}{\href{cmm12@williams.edu}{cmm12@williams.edu}}}
\address{Department of Mathematics and Statistics, Williams College, Williamstown, MA 01267}

\author[V.~Tiwari]{Vashisth~Tiwari}
\email{\textcolor{blue}{\href{vtiwari2@u.rochester.edu}{vtiwari2@u.rochester.edu}}}
\address{Department of Mathematics, University of Rochester, Rochester, NY 14627}

\author[J.~Ye]{Jingkai~Ye}
\email{\textcolor{blue}{\href{yej@whitman.edu}{yej@whitman.edu}}}
\address{Department of Mathematics and Statistics, Whitman College, Walla Walla, WA, 99362}

\author[Y.~Zhang]{Yunhao~Zhang}
\email{\textcolor{blue}{\href{mailto:lucas.zhang.12138@gmail.com}{lucas.zhang.12138@gmail.com}}}
\address{Texas Academy of Mathematics and Science, Denton, TX 76203}

\author[X.~Zheng]{Xiaoyan Zheng}
\email{\textcolor{blue}{\href{zhengxiaoyan@wustl.edu}{zhengxiaoyan@wustl.edu}}}
\address{Department of Mathematics and Statistics, Washington University in St. Louis, St. Louis, MO 63130}

\author[W.~Zhu]{Weiduo~Zhu}
\email{\textcolor{blue}{\href{mailto:amandazwd@163.com}{amandazwd@163.com}}}
\address{The Experimental High School Attached to Beijing Normal University, Beijing, China
}

\title{EXTENDING ZECKENDORF'S THEOREM TO A NON-CONSTANT RECURRENCE and THE ZECKENDORF GAME ON THIS NON-CONSTANT RECURRENCE RELATION}

\date{\today}

\begin{abstract}
Zeckendorf's Theorem states that every positive integer can be uniquely represented as a sum of non-adjacent Fibonacci numbers, indexed from $1, 2, 3, 5,\ldots$. This has been generalized by many authors, in particular to constant coefficient fixed depth linear recurrences with positive (or in some cases non-negative) coefficients. In this work we extend this result to a recurrence with non-constant coefficients, $a_{n+1} = n a_{n} + a_{n-1}$. The decomposition law becomes every $m$ has a unique decomposition as $\sum s_i a_i$ with $s_i \le i$, where if $s_i = i$ then $s_{i-1} = 0$. Similar to Zeckendorf's original proof, we use the greedy algorithm. We show that almost all the gaps between summands, as $n$ approaches infinity, are of length zero, and give a heuristic that the distribution of the number of summands tends to a Gaussian.



Furthermore, we build a game based upon this recurrence relation, generalizing a game on the Fibonacci numbers. Given a fixed integer $n$ and an initial decomposition of $n= na_1$, the players alternate by using moves related to the recurrence relation, and whoever moves last wins. We show that the game is finite and ends at the unique decomposition of $n$, and that either player can win in a two-player game. We find the strategy to attain the shortest game possible, and the length of this shortest game. Then we show that in this generalized game when there are more than three players, no player has the winning strategy. Lastly, we demonstrate how one player in the two-player game can force the game to progress to their advantage.
\end{abstract}

\keywords{Zeckendorf decompositions, Recurrence relations with non-constant coefficients, Zeckendorf game on this Relation}

\date{\today}

\thanks{This work was supported by NSF Grants DMS1561945 and DMS1947438, the Eureka Program, the University of Michigan, and Williams College.}

\maketitle
\tableofcontents

\section{Introduction}

Zeckendorf \cite{Ze} showed that every positive integer can be uniquely represented as a sum of non-adjacent Fibonacci numbers. Many related results on Zeckendorf decompositions, including uniqueness, existence, and Gaussian distribution of the number of summands, have been proven; see for example \cite{Bes,Bow,Br,Day,Dem,FGNPT,Fr,GTNP,Ha,Ho,HW,Ke,Lek,LM1,LM2,MW1,MW2,Ste1,Ste2} and the references therein. Additionally, this theorem has been generalized to a class of homogeneous linear recurrences known as positive linear recurrence sequences (see for example \cite{KKMW}). Additionally, Baird-Smith, Epstein, Flint and Miller \cite{BEFM1, BEFM2} have created a game based on the Fibonacci numbers; we show that a similar game exists for this sequence, and prove some results about it.

In this paper, we aim to achieve similar results for a particular recurrence sequence which has non-constant recurrence coefficients. The sequence in question is given by $a_{n+1} = n a_{n} + a_{n - 1}$ with initial conditions $a_1 = 1$ and $a_2 = 2$. We are concerned with a particular kind of ``legal'' decomposition, defined analogously to legal decompositions for positive linear recurrence sequences.

\begin{definition}[Legal Decomposition]\label{DD}
	A legal decomposition is a sum of the form $\sum_{i = 1}^m s_i a_i$, where $s_i \in \{0, 1, \ldots, i\}$ and if $s_i = i$, then $s_{i - 1} = 0$.
\end{definition}

This definition ensures that we cannot use the recurrence relation to replace some terms of the decomposition. Our first result is that legal decompositions exist for all positive integers.

\begin{theorem}\label{existUnique}
	There exists a legal decomposition of every positive integer into terms of the sequence $\{a_n\}$.
\end{theorem}

This can be proved in a similar fashion to the corresponding result for positive linear recurrence sequences. We then establish an explicit method for computing these decompositions using the greedy algorithm. This provides another proof of uniqueness.

\begin{theorem}\label{greedy}
	If $a_n \leq x < a_{n + 1}$, then the coefficient of $a_n$ in the legal decomposition of $x$ is $\lf x / a_n \rf$.
\end{theorem}

Next, we move on to examine the gaps between each summand and prove that most gaps in the decompositions of integers in $[a_n, a_{n + 1})$ will be of length $0$ as $n \to \infty$.

\begin{theorem}\label{th:gaps}
	As $n \to \infty$, the proportion of gaps of non-zero length in the decomposition of $m \in [a_n, a_{n + 1})$ goes to $0$.
\end{theorem}

Finally, we conjecture that the frequency of the number of summands in $[a_n, a_{n + 1})$ forms a Gaussian distribution as $n \to \infty$.

Our second set of results concerns the aforementioned game on the Fibonacci numbers, extended to this sequence. In \cite{BEFM1, BEFM2} the authors proved that in a two player game, if the input number $n$ is at least 2 then player two has a winning strategy, though the proof is non-constructive. They also proved upper and lower bounds on the length of all games, which differed by a logarithm (recent work \cite{LLMMSXZ} has removed that factor, and now the bounds are of the same order). A motivation to study this sequence was to see what results translate to this setting. In Sections \ref{sec:introgame} and \ref{sec:winninggame}, we introduce the game and state our results on the length of game and on multi-person generalizations.

\section{Proving Existence and Uniqueness of Legal Decompositions}
The goal of this section is to prove Theorem~\ref{existUnique}. We separate the proof into two parts: existence and uniqueness.

\subsection{Existence}
We prove existence by strong induction on $x$. For $x = 1$, there is the decomposition $1 \cdot a_1$. Now suppose that $x$ is a positive integer larger than 1 and that all positive integers smaller than $x$ have legal decompositions. If $x = a_i$ for some $i$, then $1 \cdot a_i$ is a legal decomposition of $x$. If $x$ is not a term of the sequence, then there exists a unique positive integer $n$ such that $a_n < x < a_{n + 1}$. Let $s_n = \lf x / a_n \rf$ and $b = x - s_n a_n$. We have
\begin{equation}
	b \ <\ x - \lp \frac{x}{a_n} - 1 \rp a_n \ =\ a_n \ <\ x,
\end{equation}
so $b$ has a legal decomposition by the inductive hypothesis. Moreover, this legal decomposition does not use $a_n$ because $a_n > b$. In other words, the decomposition takes the form
\begin{equation}
	b \ =\ \sum_{i = 1}^{n - 1} s_i a_i.
\end{equation}
By construction of $b$, we now have
\begin{equation}\label{legalDecompX}
	x \ =\ b + s_n a_n \ =\ \sum_{i = 1}^n s_i a_i.
\end{equation}
So, to finish our proof of existence, it suffices to show that $s_n \leq n$ and if $s_n = n$, then $s_{n - 1} = 0$. If $s_n > n$, then as $s_n = \lf x / a_n \rf$,
\begin{equation}
	x \ \geq\ s_n a_n \ \geq\ (n + 1) a_n \ =\ na_n + a_n \ >\ na_n + a_{n - 1} \ =\ a_{n + 1}.
\end{equation}
But by construction, we have $x < a_{n + 1}$, so $s_n \leq n$. Finally, if $s_n = n$, then
\begin{equation}
	b \ =\ x - s_n a_n \ =\ x - n a_n \ <\ a_{n + 1} - n a_n \ =\ n a_n + a_{n - 1} - n a_n \ =\ a_{n - 1},
\end{equation}
so the decomposition of $b$ cannot include $a_{n - 1}$ and $s_{n - 1} = 0$, as desired. Thus, \eqref{legalDecompX} is a legal decomposition of $x$, completing our proof.

\subsection{Uniqueness}
Before proving uniqueness, we first determine the largest integer which can be decomposed using the terms $a_1, \ldots, a_n$.

\begin{lemma}\label{largestDecomposable}
	The largest positive integer which can be legally decomposed by the terms $a_1, \ldots, a_n$ is $a_{n + 1} - 1$.
\end{lemma}

\begin{proof}
    We prove by strong induction on $n$ that if $x = \sum_{i = 1}^n s_i a_i$ is a legal decomposition, then $x < a_{n + 1}$. For $n = 1$, the only legal decomposition is $1 \cdot a_1 = 1 < 2 = a_2$, so the base case holds. Now assume the lemma holds for all $n' < n$, and let $x = \sum_{i = 1}^n s_i a_i$ be a legal decomposition. If $s_n < n$, then $x' = \sum_{i = 1}^{n - 1} s_i a_i$ is also a legal decomposition, so by the inductive hypothesis, $x' < a_n$. Thus,
    \begin{equation}
        x \ =\ \sum_{i = 1}^n s_i a_i \ <\ s_n a_n + a_n \ \leq\ (n - 1) a_n + a_n \ =\ n a_n \ <\ a_{n + 1}.
    \end{equation}
    If $s_n = n$, then $s_{n - 1} = 0$ by the definition of a legal decomposition, so $x'' = \sum_{i = 1}^{n - 2} s_i a_i$ is a legal decomposition. By the inductive hypothesis, this implies that $x'' < a_{n - 1}$, so
    \begin{equation}
        x \ =\ \sum_{i = 1}^n s_i a_i \ <\ s_n a_n + a_{n - 1} \ =\ n a_n + a_{n - 1} \ =\ a_{n + 1}.
    \end{equation}
    In either case, we see that $x < a_{n + 1}$, so the induction is complete.
\end{proof}

Now, we can prove uniqueness by showing that if two legal decompositions have the same sum, then they are the same decomposition. Assume for contradiction that there are two distinct legal decompositions with the same sum:
\begin{equation}\label{decompsWithSameSum}
	\sum_{i = 1}^n s_i a_i \ =\ \sum_{j = 1}^m t_j a_j.
\end{equation}
If $n \neq m$, then without loss of generality we may assume that $n < m$. Then by Lemma~\ref{largestDecomposable},
\begin{equation}
	\sum_{i = 1}^n s_i a_i \ <\ a_{n + 1} \ \leq\ a_m \ \leq\ \sum_{j = 1}^m t_j a_i,
\end{equation}
contradicting \eqref{decompsWithSameSum}. Thus, $n = m$, and we will use only $n$ for the remainder of the proof, as well as indexing with $i$.

Now, we want to show that $s_i = t_i$ for $i = 1, \ldots, n$. Define
\begin{equation}
	s'_i \ =\ s_i - \min(s_i, t_i) \qquad \text{and} \qquad t'_i \ =\ t_i - \min(s_i, t_i).
\end{equation}
for $i = 1, \ldots, n$. Then for each $i$, we have subtracted the same number copies of $a_i$ from both decompositions, so
\begin{equation}\label{newCoeffEqualSums}
	\sum_{i = 1}^n s'_i a_i \ =\ \sum_{i = 1}^n t'_i a_i.
\end{equation}
Additionally, for each $i$, at least one of $s'_i$ and $t'_i$ is zero because either $s_i$ or $t_i$ has been subtracted from itself in the construction of $s'_i$ and $t'_i$. If $s'_{i_1} \neq 0$ and $t'_{i_2} \neq 0$ for some maximal such $i_1$ and $i_2$, then by the same argument as above, $i_1 = i_2$. But this means that neither $s'_{i_1}$ nor $t'_{i_2}$ are zero, a contradiction. Thus, either $s_i = 0$ for all $i$, or $t_i = 0$ for all $i$. In either case, the sums in \eqref{newCoeffEqualSums} must be equal to zero, implying that $s'_i = t'_i = 0$ for $i = 1, \ldots, n$ because $a_i > 0$ and $s'_i, t'_i \geq 0$. By the construction of $s'_i$ and $t'_i$, this implies that $s_i = t_i$ for $i = 1, \ldots, n$, as desired.\qed

\section{Computing Legal Decompositions with the Greedy Algorithm}
We now can prove Theorem~\ref{greedy}, using the greedy algorithm to compute legal decompositions. Using Lemma~\ref{largestDecomposable}, the desired result follows quickly by considering what happens if the largest coefficient in the decomposition is either too large or too small.

\begin{proof}[Proof of Theorem~\ref{greedy}]
	Let $s_i$ be the number of copies of $a_i$ in the decomposition of $x$. If $s_i > \Lf x / a_i \Rf$, then
	\begin{equation}
		s_i a_i \ \geq\ \lp \Lf \frac{x}{a_i} \Rf + 1 \rp a_i \ >\ \frac{x}{a_i} a_i \ =\ x,
	\end{equation}
	which is a contradiction because $s_i a_i$ is part of the decomposition of $x$. If $s_i < \Lf x / a_i \Rf$, then
	\begin{equation}
		x - s_i a_i \ \geq\ x - \lp \Lf \frac{x}{a_i} \Rf - 1 \rp a_i \ \geq\ x - \lp \frac{x}{a_i} - 1 \rp a_i \ =\ a_i.
	\end{equation}
	Applying Lemma~\ref{largestDecomposable}, this implies that there is no legal decomposition of $x - s_i a_i$ using only the terms $a_1, \ldots, a_{i - 1}$. But one must exist since this decomposition forms the rest of the decomposition of $x$. As $s_i$ is neither greater nor less than $\Lf x / a_i \Rf$, we conclude that $s_i = \Lf x / a_i \Rf$.
\end{proof}

By iterating this theorem, we can not only compute the coefficient of the largest term in the decomposition of a given integer, but all coefficients. To do this, we find the coefficient of the largest term, then the coefficient of the second largest term, and so on by repeatedly applying Theorem~\ref{greedy} and then subtracting the newly found part of the decomposition.

\begin{example}
    We apply this process to find the legal decomposition of $x = 33$.

    The first five terms of the sequence are $1, 2, 5, 17, 73$. So the largest term in the decomposition of $33$ will be $a_4 = 17$. By repeatedly computing the coefficient of the largest term smaller than $x$, then updating $x$, we can compute all the coefficients:
    \begin{alignat*}{2}
        \Lf \frac{33}{17} \Rf \ =\ &\textcolor{red}{1} &&\quad\to\quad x \ =\ 33 - 1 \cdot 17 \ =\ 16\\
        \Lf \frac{16}{5} \Rf \ =\ &\textcolor{red}{3} &&\quad\to\quad x \ =\ 16 - 3 \cdot 5 \ =\ 1\\
        \Lf \frac{1}{1} \Rf \ =\ &\textcolor{red}{1} &&\quad\to\quad x \ =\ 2 - 1 \cdot 2 \ =\ 0.
    \end{alignat*}
    Note that we skipped $a_2 = 2$ because it was never the largest term smaller than $x$. This corresponds to the fact that the coefficient of $a_3$ is $3$, so the coefficient of $a_2$ must be zero by the decomposition rule. In all, we have the decomposition
    \begin{equation}
        33 \ =\ 1 \cdot a_4 + 3 \cdot a_3 + 1 \cdot a_1.
    \end{equation}
\end{example}

\section{The Distribution of Gaps between Summands}

The distribution of gaps has previously been studied in the context of generalized Zeckendorf decompositions by many authors, see for example \cite{BBGILMT, Bow, LM2}. \emph{Gaps} are differences in indices between each pair of adjacent summands, including identical ones, in the decomposition. Two identical summands constitute a \emph{gap of length zero,} and the gaps between distinct summands are \emph{non-zero gaps.} Additionally, the length of non-zero gaps depends on the difference of indices between the two adjacent distinct summands.

As $n$ grows, almost all of the gaps are zero. This is due to how rapidly our sequence grows; most indices are used multiple times (on average index $i$ occurs about $i/2$ times in a typical decomposition, with fluctuations on the order of $\sqrt{i}$, leading to a large number of gaps of length zero). The only way to get a gap of length 1 or more is from distinct summands, and there are at most $n$ such opportunities. We now prove Theorem~\ref{th:gaps}.

\begin{proof}[Proof of Theorem~\ref{th:gaps}]
We set $I(n):= [a_n, a_{n+1})$ to be the interval we are studying, and $A_n := a_{n+1}-a_n$ the number of terms of our sequence in that interval.
First, we show that $A_n \le (n+1)!$. To see this, note all numbers in $I(n)$ are of the form $\sum_{i=1}^n s_i a_i$, where $s_n \ge 1$, $s_i \in \{0, 1, \dots, i\}$ and if $s_i = i$ then $s_{i-1} = 0$. If we drop the last condition we obtain an upper bound; for each $i$ there are now $i+1$ choices, and thus $A_n \le (n+1)!$.

Next, for each $m \in I(n)$ there can be at most $n$ non-zero gaps. Thus the number of non-zero gaps arising from decompositions of numbers in $I(n)$ is at most $n \cdot (n+1)!$.

We now show that there are tremendously more gaps of length zero. As we are considering behavior in the limit, we may consider the subset of numbers of the form $\sum_{i=16}^n t_i a_i$, where $t_i \in \{\lfloor i/4\rfloor + 2, \dots, \lfloor 3i/4\rfloor + 3\}$ (we choose 16 to avoid any edge effects; we do not want for example the upper bound to exceed $n$). Note we have at least $i/2 + 1$ choices for each $i$, each $t_i \geq \lfloor i/4\rfloor + 2$ so each choice generates at least $i/4$ gaps of length zero, and all of these are legal decompositions of integers in $I(n)$ as no $t_i = i$. The number of such numbers is at least $\prod_{i=16}^n i/2$, which is $C n!/2^n$ for some fixed $C$. As each of these numbers generates at least $\prod_{i=16}^n (i/4) = C n!/4^n$ gaps of length zero, we see the total number of gaps of length zero is at least $C^2 n!^2 / 8^n$. This is tremendously larger than the number of non-zero gaps, as $n!/8^n \ge (n/8e)^n \ge n^3$, which implies that $C^2 n!^2 / 8^n > n \cdot (n+1)!$, for large enough $n$. Thus in the limit almost all gaps are of length zero.

\end{proof}

\section{Gaussianity}

Many researchers \cite{BBGILMT, Bow} have studied the distribution of the number of summands and the gaps between summands in generalized Zeckendorf decompositions. In positive linear recurrence systems, as well as some other systems, the answers have been found to be a Gaussian and geometric decay. One of the reasons we chose to study this non-constant coefficient recurrence was to see if these behaviors persist.

We conjecture that the distribution of the number of summands for $m \in [a_n, a_{n+1})$ converges to a Gaussian as $n \to \infty$. We hope to return to this in future work, though numerical studies strongly support this, as do results for a similar system. In particular, if we drop the assumption that if we have $i$ copies of $a_i$ then we must have 0 copies of $a_{i-1}$, Gaussianity follows immediately from Lindeberg's Central Limit Theorem (see \cite{Li, Za}).


\section{Introduction to the Generalized Zeckendorf Game}\label{sec:introgame}
Zeckendorf proved that every positive integer $n$ can be written uniquely as the sum of non-adjacent Fibonacci numbers, now known as the Zeckendorf decomposition of $n$. Baird-Smith, Epstein, Flint and Miller \cite{BEFM1,BEFM2} create a game based on the Zeckendorf decomposition. Zeckendorf's theorem has been generalized to the non-constant recurrence relation $a_{i+1}  = i\:a_{i}\:+\:a_{i-1}$ in Theorem \ref{existUnique}, allowing a game to be based on this recurrence.

We introduce some notation. By $\{1^n\}$ or $\{a_1^n\}$ we mean $n$ copies of $1$, the first number in the sequence. If we have 3 copies of $a_1$, 2 copies of $a_2$, and 7 copies of $a_4$, we could write either $\{a_1^3 \wedge a_2^2 \wedge a_4^7 \}$ or $\{1^3 \wedge 2^2 \wedge 17^7\}$.

\subsection{Definition of the Game}
\begin{defi}
 Let $a_1 = 1$, $a_2 = 2$, and $a_{i+1} = i\:a_{i}\:+\:a_{i-1}$. At the beginning of the game, there is an unordered list of $n$ $1$'s. We denote the initial list as $\{a_1^n\}$ where $n\in\mathbb{N}=\{1,2,3,\dots\}$. On each turn, a player can do one of the following moves which are based on the recurrence:
\begin{enumerate}
    \item  Combining moves:
    \begin{enumerate}
        \item If the list contains consecutive terms $a_i$ and $a_{i-1}$ such that there are at least $i$ $a_i$'s and one $a_{i-1}$, we can combine these to create $a_{i+1}$. This move is denoted by $\{a_{i}^ i\wedge a_{i-1}\rightarrow a_{i+1}\}$.
        \item If the list contains two $1$'s, we can combine $1$'s. This move is denoted by $\{1^2 \rightarrow 2\}$.
    \end{enumerate}
     \item Splitting moves:
    \begin{enumerate}
        \item
        Note that
        \begin{align*}
           (i+1)a_{i} &= i\:a_{i}\:+\:a_{i} \\
           &= i\:a_{i}\:+\:(i-1)a_{i-1}+a_{i-2}\\
           &= a_{i+1}+(i-2)a_{i-1}+a_{i-2}.\\
        \end{align*}
        Thus if the list contains $(i+1)$
        $a_{i}$'s, we can we can perform a splitting move in the following manner:
        $\{a_{i}^{i+1}\rightarrow a_{i+1}\wedge a_{i-1}^{i-2}\wedge a_{i-2}\}$.
        \item If the list contains three $2$'s, we can perform a splitting move denoted by $\{2^3 \rightarrow 1\wedge5\}$.
    \end{enumerate}
\end{enumerate}
The players alternate moving until no moves remain.
\end{defi}

 The game can have any number of players, $p$, for $p\in \mathbb{N}$.
 We will show that this game is finite and ends when the list is exactly the unique legal decomposition of $n$ ($n = \sum s_i a_i$, $0\leq s_i\leq i$), as at this point there are no possible moves left to be made. The player who makes the last move wins the game.

 Figure \ref{n=10 general game} shows a two-player sample game tree for $n=10$.
\begin{figure}[h]
    \centering
    \includegraphics[scale=0.5]{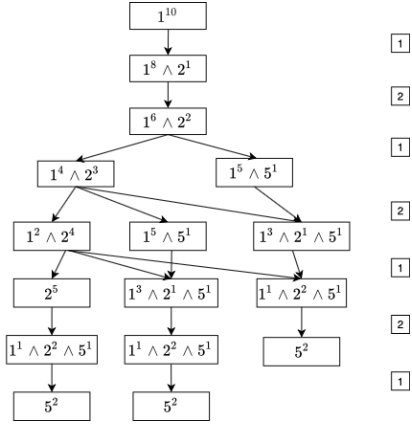}
    \caption{Game tree for $n=10$, showing all possible moves and where the game ends for each set of moves. Note that the game ends at the unique decomposition of $10$ which is given by $\{5^2\}$ (two copies of $5$).}
    \label{n=10 general game}
\end{figure}

\subsection{Properties of the Game}

\begin{thm}
\textbf{The game is finite:} Every game played on $n$ terminates within a finite number of moves at the unique decomposition given by $n = \sum s_i a_i$, $0\leq s_i\leq i$, where $a_i$ is the $i^{\text{th}}$ term in the sequence defined by $a_{i}  = (i-1)\:a_{i-1}\:+\:a_{i-2}$.
\end{thm}

\begin{proof}
Consider the number of terms in the game. We show that this number is a strictly decreasing monovariant.\\
Our moves cause the following changes in the proposed monovariant. We observe that we only have to consider the terms affected by each move because the suggested monovariant is a sum, so unaffected terms contribute the same before and after the move. Here, $i$ is the index of $a_i$, a term in the current game state.
\begin{enumerate}
    \item Combining $1$'s: The move is characterized by $\{1^2 \rightarrow 2\}$. Thus we go from having  $2$ terms to $1$ term.
    \item Combining consecutive terms: This move is characterized by $\{a_{i}^ i\wedge a_{i-1}\rightarrow a_{i+1}\}$. Thus, the number of terms goes from $i+1$ terms to $1$ term.
    \item  Splitting moves: The splitting moves are given by $\{2^3 \rightarrow 1\wedge5\}$ and $\{a_{i}^{i+1}\rightarrow a_{i+1}\wedge a_{i-1}^{i-2}\wedge a_{i-2}\}$ respectively. Note that for all $i$, splitting moves cause the number of terms to go from $i+1$ terms to $i$ terms.
\end{enumerate}

We see that every move decreases the number of terms in the game at any state.
The game progresses along a subset of the partitions of $n$ and must end at the legal decomposition of $n$, for if it did not, there would still be terms $a_i$ such that we have $(i+1)$ of them, or the recurrence would apply, by definition. Hence there would still be a combining or splitting move possible. From this we know we must start with $n$ terms and end with $LZ(n)$ terms, where $LZ(n)$ is the number of terms in the legal decomposition of $n$. Therefore, since each move decreases the number of terms by at least $1$, the game can take at most $n-LZ(n)$ moves to complete, thus is finite.
\end{proof}

Now that we know that this game does indeed end in finitely many moves, this leads us to wonder how many moves must be played to finish the game. But first, we address whether it is possible for either player in a two-player game to win.

\begin{thm}
\textbf{The game can be won by either player in a two-player game:} For $n\geq6$, there are at least two games with different numbers of moves, where at least one game has an odd number of moves and one has an even number of moves.
\end{thm}

\begin{proof}
We show using the game on $n=6$ that the game on $n\geq6$ can end in either an even or an odd number of moves, indicating that either player can win the game.

Let $n\geq 6$ and let the game begin with either of the following sequences of moves to first decompose $6$:
\begin{enumerate}
    \item $M_1 =\{\{1\wedge 1 \rightarrow 2\},\{1\wedge1\rightarrow2\},\{1 \wedge2 \wedge 2 \rightarrow5\}\}$ {(3 steps, $|M_1|=3$)},
    \item $M_2 =\{\{1\wedge 1 \rightarrow 2\},\{1\wedge1\rightarrow2\},\{1 \wedge1 \rightarrow2\},\{2\wedge2\wedge2 \rightarrow 5\wedge1\}\}$ {(4 steps, $|M_2|=4$)}.
\end{enumerate}
Now, let the set of moves it takes to resolve the rest of the game be $M_k$ with $|M_k| = k$. Regardless of what $k$ is, there are two sets of moves with different parities, $M_1\wedge M_k$ and $M_2\wedge M_k$, that describe a complete game.\\
For $k$ odd, $|M_1\wedge M_k| = 3 + k$ will be even and $|M_2\wedge M_k| =\:4 + k$ will be odd, and vice versa for $k$ even.\\
Therefore, for $n\geq6$ there exists at least one game with an even number of moves and one with an odd number of moves, giving both players a chance of winning the game.
\end{proof}

Note that this proof only addresses whether it is possible for either player to win, not that any player has the winning strategy. Later on, we discuss who may or may not have the winning strategy in games of multiple players, and in the two-player case, a strategy for some $n$.

\subsection{The Game with Only Combining Moves}
We now investigate this generalized Zeckendorf game where only combining moves are performed. We show that in this type of game the least amount of moves are performed compared to any other game on $n$.
\\

For this section we use the following notation:

$a_i$: The $i^{\text{th}}$ term in the generalized sequence.

$\delta_i$: The coefficient of the $a_i$ in the final decomposition of $n$.

$k$: The largest index in the unique decomposition of $n$, thus the decomposition is written as: \begin{equation}
n\ =\ \delta_1a_1\ +\ \delta_2a_2\ +\ \cdots\ +\ \delta_ka_k.
\end{equation}

$C_i$: A combining move on $a_i$, i.e., $$\{a_i^{i}\wedge a_{i-1}\rightarrow a_{i+1}\}.$$

$S_i$: A splitting move on $a_i$, i.e., $$\{a_i^{i+1}\rightarrow a_{i+1}\wedge a_{i-1}^{i-2}\wedge a_{i-2}\}.$$

$MC_i$: The total number of $C_i$ moves performed in a game on $n$.

$MS_i$: The total number of $S_i$ moves performed in a game on $n$.

$MC(n)$: The sum of all $MC_i$ (for $1\leq i \leq k$) performed in a game on $n$.


\begin{lem}\label{combineGameExists}
For any $n\in\mathbb{N}$, it is possible to play the game on just combining moves.
\end{lem}

\begin{proof}
We first show that this is true for any term in the sequence by inducting on the index of the $a_i$.

\textit{Base cases:}
$i=1$: We play the game on $a_{1}=1$. This game has 0 moves.

$i=2$: We play the game on $a_{2}=2$, which consists of one combining move: perform $C_1$ by combining two $1$'s to get one $2$.

$i=3$: We play the game on $a_{3}=5$, which consists of three combining moves: perform $C_1$ twice to get two $2$'s, then $C_2$ by combining two $2$'s with one $1$ to get $5$.

\textit{Inductive step:} Suppose for all $a_{i},\ i<j$ for some $j\in\mathbb{N}$, the game on $a_{i}$ can be played using only combining moves. Since $a_j=(j-1)a_{j-1}+a_{j-2}$, all that needs to be done is to perform the combining moves necessary to get $(j-1)$ $a_{j-1}$'s and one $a_{j-2}$, then perform a $C_{j-1}$ move to get one $a_{j}$.

Since an arbitrary $n$ has the decomposition $$n\ =\ \delta_1 a_{1}\ +\ \delta_2 a_{2}\ +\ \cdots\ +\ \delta_k a_{k},$$
a game with all combining moves can be played by achieving first $\delta_k a_{k}$, then $\delta_{k-1} a_{k-1}$, and so on as described above until the decomposition is achieved.
\end{proof}

\begin{thm}\label{combining}
The total number of combining moves, $MC(n)$, for a game on $n$ is a constant independent of how the game is played.
\end{thm}
\begin{proof}
We show this using a system of equations for the final coefficients of the $a_i$ in the decomposition, $\delta_i$, in terms of the $MC_i$ and $MS_i$. For $\delta_1$, note that at the beginning of the game we start with $n$ $1$'s. Every $C_1$ move decreases the amount of $1$'s by two, and every $C_2$ move decreases the amount by one. Every $S_2$ and $S_3$ move increases the amount of $1$'s by one. For $\delta_2$, note that at the beginning of the game we start with zero $2$'s. Every $C_1$ move increases the amount of $2$'s by one, every $C_2$ move decreases the amount by two, and every $C_3$ move decreases the amount by one. Every $S_2$ move decreases the amount of $2$'s by three and every $S_3$ and $S_4$ move increase the amount by one. Hence we have the following equations for $\delta_1$ and $\delta_2$:
\begin{align}
   \delta_1 &= n-2MC_1-MC_2+MS_2+MS_3,\nonumber\\
   \delta_2 &= MC_1-2MC_2-MC_3-3MS_2+MS_3+MS_4.
\end{align}
For $3\leq i\leq k$ ($k$ being the largest index in the final decomposition), every $C_{i-1}$ move increases the amount of $a_i$ by $1$, $C_{i}$ decreases the amount by $i$, and $C_{i+1}$ decreases the amount by $1$. As for splitting moves, every $S_{i-1}$ increases the amount of $a_i$ by $1$, $S_i$ decreases the amount by $i+1$, $S_{i+1}$ increases the amount by $i-1$, and $S_{i+2}$ increases the amount by $1$. Thus for $\delta_i$ we have the equation
 \begin{align}
    \delta_i &= MC_{i-1}-iMC_{i}-MC_{i+1}+MS_{i-1}-(i+1)MS_i+(i-1)MS_{i+1}+MS_{i+2}.
\end{align}

Note that $C_k=C_{k+1}=S_k=S_{k+1}=S_{k+2}=0$, so they will not be variables in our system of equations. With this system of equations we produce a matrix which we will use to prove Theorem \ref{combining}.

Let $M=\begin{pmatrix}
\begin{bmatrix}
A\end{bmatrix}\begin{bmatrix}
B\end{bmatrix}\end{pmatrix}$
where
\begin{equation}\label{matrixA}
    A\ =\
\begin{bmatrix}
1&-2&-1&0&0&0&\cdots&&\cdots&0\\
0&1&-2&-1&0&0&\cdots&&\cdots&0\\
0&0&1&-3&-1&0&\cdots&&\cdots&0\\
\vdots&&&\ddots&\ddots&\ddots&&&&\vdots\\
0&\cdots&\cdots&0&1&-i&-1&0&\cdots&0\\
\vdots&&&&&\ddots&\ddots&\ddots&&\vdots\\
0&\cdots&&&\cdots&0&1&3-k&-1&0\\
0&\cdots&&&\cdots&0&0&1&2-k&-1\\
0&\cdots&&&\cdots&0&0&0&1&1-k\\
0&\cdots&&&\cdots&0&0&0&0&1\\
\end{bmatrix}
\end{equation}
and \begin{equation}
B\ =\
\begin{bmatrix}
1&1&0&0&0&0&\cdots&&\cdots&0\\
-3&1&1&0&0&0&\cdots&&\cdots&0\\
1&-4&2&1&0&0&\cdots&&\cdots&0\\
0&1&-5&3&1&0&\cdots&&\cdots&0\\
\vdots&&\ddots&\ddots&\ddots&\ddots&&&&\vdots\\
0&\cdots&0&1&-i-1&i-1&1&0&\cdots&0\\
\vdots&&&&\ddots&\ddots&\ddots&\ddots&&\vdots\\
0&\cdots&&\cdots&0&1&3-k&k-5&1&0\\
0&\cdots&&\cdots&0&0&1&2-k&k-4&1\\
0&\cdots&&\cdots&0&0&0&1&1-k&k-3\\
0&\cdots&&\cdots&0&0&0&0&1&-k\\
0&\cdots&&\cdots&0&0&0&0&0&1\\
\end{bmatrix}.
\end{equation}\\
Note $A$ is the $k\times k$ invertible submatrix of the $n$ and $MC_i$ terms, and $B$ is the $k\times (k-2)$ submatrix of the $MS_i$ terms.

For the vectors
\begin{equation}\label{v and delta}
\delta = \begin{pmatrix}
\delta_1\\\delta_2\\\vdots\\\delta_k
\end{pmatrix},\ v = \begin{pmatrix}
n\\MC_1\\\vdots\\MC_{k-1}\\MS_2\\\vdots\\ MS_{k-1}\\
\end{pmatrix},
\end{equation}
we have $Mv\ =\ \delta$. To find an expression for the total number of moves in the game, we must multiply
\begin{equation}
\begin{pmatrix}
0&1&1&\cdots&1
\end{pmatrix}v.
\end{equation}

In reduced row echelon form, the equation $Mv\ =\ \delta$ is as follows:
\begin{equation}
\begin{pmatrix}
\begin{bmatrix}
I_k
\end{bmatrix}
\begin{bmatrix}
0&0&0&0&\cdots&0\\
-1&0&0&0&\cdots&0\\
1&-1&0&0&\cdots&0\\
0&1&-1&0&\cdots&0\\
\vdots&&\ddots&\ddots&&\vdots\\
0&\cdots&0&1&-1&0\\
0&\cdots&0&0&1&-1\\
0&\cdots&0&0&0&1\\
\end{bmatrix}
\end{pmatrix}
\begin{pmatrix}
n\\MC_1\\\vdots\\MC_{k-1}\\MS_2\\\vdots\\ MS_{k-1}\\
\end{pmatrix}
= \begin{pmatrix}
\delta_1\\\delta_2\\\vdots\\\delta_k
\end{pmatrix}.
\end{equation}
From this we can see that the $n$ and $MC_i$ terms are pivot variables, and the $MS_i$ are free variables. With this we can solve for $v$:
\begin{equation}
\begin{pmatrix}
n\\MC_1\\\vdots\\MC_{k-1}\\\hline MS_2\\\vdots\\ MS_{k-1}\\
\end{pmatrix}
=\begin{pmatrix}
A^{-1}\begin{pmatrix}
\delta_1\\\delta_2\\\vdots\\\delta_k\\
\end{pmatrix}\\\hline
0\\\vdots\\0
\end{pmatrix}\ +\ MS_2\begin{pmatrix}
0\\1\\-1\\0\\0\\\vdots\\\hline 1\\0\\0\\\vdots\\0
\end{pmatrix}\ +\ MS_3\begin{pmatrix}
0\\0\\1\\-1\\0\\\vdots\\\hline 0\\1\\0\\\vdots\\0
\end{pmatrix}\ +\\\cdots\ +\  MS_{k-1}\begin{pmatrix}
0\\0\\\vdots\\0\\1\\-1\\\hline 0\\0\\\vdots\\0\\1
\end{pmatrix}.
\end{equation}
We then need to multiply the right hand side by $\begin{pmatrix}
0&1&1&\cdots&1\end{pmatrix}$. Thus the total number of moves is given by
\begin{equation}
\begin{pmatrix}
0&1&1&\cdots&1\end{pmatrix}A^{-1}\begin{pmatrix}
\delta_1\\\vdots\\\delta_k
\end{pmatrix}
\ +\ MS_2\ +\ MS_3\ +\ \cdots\ +\  MS_{k-1}.
\end{equation}
Note that all $MS_i$ terms are left in their original form in the equation, but the sum of $MC_i$ terms, $MC(n)$, is now replaced with $\begin{pmatrix}
0&1&1&\cdots&1\end{pmatrix}A^{-1}\begin{pmatrix}
\delta_1\\\vdots\\\delta_k
\end{pmatrix}$. This value is based solely on the unique decomposition of $n$, thus is constant no matter how the game on $n$ is played.
\end{proof}

\begin{corollary}\label{shortestGame}
A game with only combining moves realizes the shortest game for all $n$.
\end{corollary}
\begin{proof}
This follows directly from Theorem \ref{combining}. Since $MC(n)$ is constant for any game on $n$, performing splitting moves will increase the length of the game.
\end{proof}

\begin{thm}
On a game starting with $n$ where only combining moves are performed, the game never has more moves than $.7757n$.
\end{thm}
\begin{proof}
In a game on $n$ with no splitting moves, we have $MC_1\leq n/2$ since we need two $1$'s to perform a $C_1$ move. Likewise, $MC_2\leq n/5$ since five $1$'s are needed to perform $C_2$, and so on. More generally, $MC_{i}\leq {n}/{a_{i+1}}$, for all $i\in\mathbb{N}$. Hence for the game on $n$ the combining moves are bounded as:
\begin{equation}
    MC(n)\ \leq\  n \sum_{i=1}^{k} \frac{1}{a_{i+1}}.
\end{equation}

We now prove that
\begin{equation}\label{originalinequality}
      \sum_{i=1}^{k} \frac{1}{a_{i+1}}\ < 0.7757.
\end{equation}

Since for any $i\geq 2$,
\begin{eqnarray}
      a_{i+1} & \ =\ &ia_{i}\ +\ a_{i-1} \nonumber\\
      \frac{1}{a_{i+1}} &\ =\ & \frac{1}{ia_{i} + a_{i-1}} < \frac{1}{ia_{i}}\ \leq \ \frac{1}{2a_{i}},
\end{eqnarray}

so
\begin{equation}
      \frac{1}{a_{i+1}}\ <\ \frac{1}{2a_{i}}.
\end{equation}

We return to the proof of the original inequality (\ref{originalinequality}). For $k \geq 7$,
\begin{align}
      \sum_{i=1}^{k} \frac{1}{a_{i+1}} &= \frac{1}{2} + \frac{1}{5} + \frac{1}{17} +\frac{1}{73} + \frac{1}{382}+ \frac{1}{2365} +\frac{1}{16937}+\cdots+\frac{1}{a_{k+1}}\nonumber\\
      &< \frac{1}{2} + \frac{1}{5} + \frac{1}{17} + \frac{1}{73} + \frac{1}{382}+ \frac{1}{2365}+\frac{1}{16937}+\frac{1}{16937}\left(\frac{1}{2}+ \frac{1}{4}+\frac{1}{8}+\cdots+\frac{1}{2^{k-6}}\right)\nonumber\\
    &\ =\ \frac{1}{2} + \frac{1}{5} + \frac{1}{17} + \frac{1}{73} + \frac{1}{382}+ \frac{1}{2365} +\frac{1}{16937}+\frac{1}{16937}\left(1-\frac{1}{2^{k-6}}\right) \nonumber\\
      & \ < \ \frac{1}{2} + \frac{1}{5} + \frac{1}{17} + \frac{1}{73} + \frac{1}{382}+ \frac{1}{2365} + \frac{1}{16937}+ \frac{1}{16937}\nonumber\\
      &\ <\ 0.7757.
\end{align}
Thus we have that $\sum_{i=1}^{k} {1}/{a_{i+1}} <0.7757$,
and
\begin{equation}
    MC(n)\ \leq\  n \sum_{i=1}^{k} \frac{1}{a_{i+1}}\ <\ 0.7757n.
\end{equation}

\end{proof}

Experimental data for the value of $MC(a_i)$ for $i\in [1,100]$ suggest that this upper bound can be tightened further. As $i$ approaches $100$, the number of combining moves for a game on $a_{i}$ approaches a value around $0.6601\:a_{i}$. The exact value that $MC(a_i)$ converges to as $i\rightarrow\infty$ is currently unknown.

\subsection{The Number of Moves in a Combine Only Game}

In this section we derive the exact formula for $MC(n)$ which can be evaluated for any $n$. Note that for $a_i$, a term in the sequence, we can find $MC(a_i)$ using the following recurrence:
$$MC(a_1)\ =\ 0,\  MC(a_2)\ =\ 1,$$ and for $i\geq 3$
\begin{equation}
 MC(a_i)\ =\ (i-1)MC(a_{i-1})\ +\ MC(a_{i-2})\ +\ 1.
\end{equation}
This is due to the repetitive nature of the game with only combining moves, as demonstrated in Lemma \ref{combineGameExists}. The first several terms in this sequence are
$$0,\ 1,\ 3,\ 11,\ 48,\ 252,\ 1561,\ 11180,\ \dots.$$

The values $MC(a_i)$ ($i=1,\dots, k$) can then be used to find the value of $MC(n)$ for an arbitrary $n\in\mathbb{N}$ with decomposition $n=\delta_1a_1+\delta_2a_2+\cdots+\delta_ka_k.$

\begin{thm}\label{combineMovesFormula}
The number of combining moves in a game on $n$ with decomposition\\ $n=\delta_1a_1+\delta_2a_2+\cdots+\delta_ka_k$ is $$MC(n)\ =\ \delta_2MC(a_2)\ +\ \delta_3MC(a_3)\ +\ \cdots\ +\ \delta_kMC(a_k).$$
\end{thm}

We dedicate the rest of this section to prove Theorem \ref{combineMovesFormula}.

In a Combining Only game on any $n$, the number of moves is \begin{equation}MC(n)\ =\ MC_1\ +\ \cdots\ +\ MC_{k-1}\end{equation} where $k$ is the largest index such that $a_{k}$ is in the unique decomposition of $n$, and $C_{k-1}$ is performed at most $k$ times.

Note that for a Combine Only game we have a system of equations similar to the system utilized in the proof of Theorem \ref{combining}, except that all $MS_i$ are $0$. The $\delta_i$ ($1\leq i\leq k$) are the coefficients of the $a_i$ in the final decomposition, and are written in terms of the $MC_i$.

Since a $C_1$ move removes two $1$'s, and $C_2$ removes one, we get the following equation for $\delta_1$:
\begin{align}
\delta_1\ &=\  n-2MC_1-MC_2.
\end{align}
For $2\leq i\leq k$,
\begin{equation}\delta_{i}\ =\ MC_{i-1}-iMC_{i}-MC_{i+1}.\end{equation}
Note that $MC_{k+1}=MC_k=0$.
The system has $k$ equations and $k-1$ unknowns so we can solve for all $k-1$ unknowns, $MC_1,\dots,MC_{k-1}$, and sum them to get $MC(n)$.
We solve this system using matrices.

Let matrix $A$ be as in equation (\ref{matrixA}), the matrix of coefficients of the $n$ and $MC_i$. We also have the vector of the $\delta_i$
as defined in
(\ref{v and delta}). Finally, let $$u=\begin{pmatrix}
n\\MC_1\\\vdots\\MC_{k-1}
\end{pmatrix}.$$
These satisfy the equation $Au\ =\ \delta$. Since $A$ is invertible, we also have $A^{-1}\delta\ =\ u$.
Consider the matrix
\begin{equation}\label{matrixB}
B\ =\ \begin{bmatrix}
B_1(1)&B_1(2)&B_1(3)&B_1(4)&B_1(5)&\cdots&B_1(k)\\
0&B_2(1)&B_2(2)&B_2(3)&B_2(4)&\cdots&B_2(k-1)\\
\vdots&&\ddots&\ddots&&&\vdots\\
0&\cdots&0&B_i(1)&B_i(2)&\cdots&B_i(k-i+1)\\
\vdots&&&&\ddots&\ddots&\vdots\\
0&\cdots&&\cdots&0&B_{k-1}(1)&B_{k-1}(2)\\
0&\cdots&&\cdots&0&0&B_k(1)\\
\end{bmatrix},
\end{equation}
where the $B_i$ are sequences such that $B_1(m)=a_{m}$ and for all other $i$, $B_i(1)=1, B_i(2)=i,$ and for $m\geq3$, $B_i(m)= (m-1+i)B_i(m-1)+B_i(m-2)$, where $i$ is the row of the matrix, and the $m^{\text{th}}$ element of the sequence, $B_i(m)$, is the element in the $i^{\text{th}}$ row and the $(m+i-1)^{\text{th}}$ column of the matrix $B$. We can also write the recurrence in terms of the columns:
$B_i(1)=1,\ B_i(2)=i,\ B_i(m)=(j)B_i(m-1)+B_i(m-2)$, where $j$ denotes the column of the element $B_i(m-1)$.

It is useful to consider this recurrence solely in terms of the row and column numbers. The entry in the $i^{\text{th}}$ row and $j^{\text{th}}$ column of $B$ is
\begin{equation}
B_{i,j} \ = \
\begin{cases}
      0 &  j<i \\
      1 &  j=i \\
      j-1 &  j=i+1\\
      (j-1)B_{i,(j-1)}+B_{i,(j-2)} & j\geq i+2\ ,
\end{cases}
\end{equation}
or equivalently, \begin{equation}\label{matrix2eq}
B_{i,j} \ = \
\begin{cases}
      1 &  i=j \\
      i &  i=j-1 \\
      iB_{(i+1),j}+B_{(i+2),j} & i\leq j-2\ .
\end{cases}
\end{equation}
\begin{lemma}
For $k\times k$ matrices $A$ and $B$ as in (\ref{matrixA}) and (\ref{matrixB}) respectively, we have $AB=I_k$ and hence $B=A^{-1}$.
\end{lemma}
\begin{proof}
First we show that for all $1\leq i\leq k$, $A_{i,0}\bullet B_{0,i}=1$, where $\bullet$ represents the dot product, and $A_{i,0}$ denotes the $i^{\text{th}}$ row of $A$, and $B_{0,i}$ denotes the $i^{\text{th}}$ column of $B$. For any row $A_{i,0}$, note that the first non-zero entry is in column $i$ and is always $1$. Similarly, for any column $B_{0,i}$, the last nonzero entry is in row $i$ and is always $1$. Hence for all $i$, $A_{i,0}\bullet B_{0,i}=1$. So for the matrix product $AB = C$, the matrix $C$ has all $1$'s on the diagonal. We must now show that for all $i$, $A_{i,0}\bullet B_{0,j}=0$ for $j\neq i$. \\ \

\noindent Case 1, $j<i$: Here the last nonzero term in $B_{0,j}$ is in row $j$, where the first nonzero entry in $A_{i,0}$ is in column $i$. Hence $A_{i,0}\bullet B_{0,j}=0$. \\ \

\noindent Case 2, $j>i$:
We notice that
\begin{equation}
A_{i,0}\bullet B_{0,j} \ = \ 1(B_{i,j})+(-i)(B_{i+1,j})+(-1)(B_{i+2,j}),
\end{equation}
which is $0$ by (\ref{matrix2eq}).\\
Thus we have $B=A^{-1}$.
\end{proof}
\ \\
We then have
\begin{equation}
MC(n) \ = \ \begin{pmatrix}
0&1&1&...&1\\
\end{pmatrix}A^{-1}\delta.
\end{equation}
We first multiply $\begin{pmatrix}
0&1&1&...&1\\
\end{pmatrix}$ with $A^{-1}$. The $j^{\text{th}}$ entry in this product starting from $j=2$ (when $j=1$ the entry is $0$) is the sum
$$\sum_{i=2}^{j} B_{i,j},$$
where $B_{i,j}$ is the entry in the $i^{\text{th}}$ row and $j^{\text{th}}$ column of matrix $B$ (\ref{matrixB}). Note that $\sum_{i=2}^{2} B_{i,2}=1$ and $\sum_{i=2}^{3} B_{i,3}=3$. Using the recurrence of the $B_{i,j}$'s, it can be shown that
\begin{equation}
\sum_{i=2}^{j} B_{i,j}\ =\  (j-1)\sum_{i=2}^{j-1} B_{i,j-1}\ +\  \sum_{i=2}^{j-2} B_{i,j-2}\ +\ 1.
\end{equation}
Thus this summation follows the recurrence of the $MC(a_i)$, giving us
\begin{equation}
\sum_{i=2}^{j} B_{i,j}\ =\ MC(a_j).
\end{equation}
Hence
\begin{align*}
    \begin{pmatrix}
    0&1&1&\dots&1
    \end{pmatrix}A^{-1} & \ = \ \begin{pmatrix}
    0&1&3&11&48&\dots&MC(a_{k-1})&MC(a_{k})\end{pmatrix} \\& \ =\  \begin{pmatrix}
    MC(1)&MC(2)&MC(5)&\dots&MC(a_{k-1})&MC(a_{k})\end{pmatrix}.
\end{align*}

When multiplied with the vector $\delta$ we get
\begin{align}
MC(n) &= \begin{pmatrix}
MC(1)&MC(2)&MC(5)&\dots&MC(a_{k-1})&MC(a_{k})
\end{pmatrix}
\begin{pmatrix}
\delta_1\\
\delta_2\\
\vdots\\
\delta_k\\
\end{pmatrix}\nonumber\\
&= \delta_2MC(a_2)\ +\ \delta_3MC(a_3)\ +\ \cdots\ +\ \delta_kMC(a_k).
\end{align}
This expression gives us the number of moves in a Combine Only game, or equivalently, the number of combining moves in any game, proving Theorem \ref{combineMovesFormula}. By Corollary \ref{shortestGame}, it is the exact length of the shortest game on $n$.

\section{Winning Strategies}\label{sec:winninggame}
\begin{thm}\label{T6}
When there are at least $4$ players $(p\geq 4)$ in the Generalized Zeckendorf game  and $n\geq 16$, no player has a winning strategy.
\end{thm}

\begin{thm}\label{T7}
In a $3$-player Generalized Zeckendorf game $(p=3)$, for any $n \geq 5$, player $2$ will never have a winning strategy.
\end{thm}

\begin{thm}\label{T8}
For any significantly large $n$ $(n \geq 4m^2+8m)$, when there are two alliances, with one having $m$ consecutive players, and the other having $3m$ consecutive players (which we will term the big alliance), then the big alliance always has a winning strategy.
\end{thm}

Recall that $C_i$ represents the combining move $\{a_{i-1} \wedge i \: a_i \rightarrow a_{i+1}\}$, and when $i=1$, $\{1\wedge1\rightarrow2\}$. We define $S_i$ as the splitting move requiring $(i+1)\: a_i$.

\begin{thm}\label{T9}
In a two-player game, as long as there are 1's remaining in the game state, one player making the first available move of $C_3$, $C_2$, $C_4$, $C_5$, $\ldots$, $C_{k}$\footnote{Where $C_{k}$ is the largest combining move that can be made in a game starting with $n$ $1$'s.}, $C_1$ will be able to force the game to progress without a splitting move being made.
\end{thm}

\subsection{Multiplayer Games, \texorpdfstring{$p>2$}{p>2}}

To prove Theorem \ref{T6}, we utilize the following property.

\textbf{\emph{Property 1.}}
Suppose player $m$ has a winning strategy $(1\leq m \leq p)$. For any $p\geq 4$ and $n$ significantly large, any winning path of player $m$ does not contain the following $4$ consecutive steps listed below unless player $m$ is the player who takes Step $3$:

Step $1: 1+1=2$ $($combining two $1$'s into one $2)$

Step $2: 1+1=2$

Step $3: 1+1=2$

Step $4: 2+2+2=1+5$ $($splitting three $2$'s into one $1$ and one $5)$.

\begin{proof}
Suppose player $m$ is not the player who takes Step $3$.
Suppose also that player $m$ has a winning strategy and there is a winning path consisting of the four steps listed above. Then the player in Step $3$ can take $1+2+2=5$ instead and keep the rest of the steps after the original Step $4$ exactly the same.

So now player $m-1$ has the winning strategy, which contradicts our assumption that player $m$ has the winning strategy. The property is proved by stealing the ``winning'' strategy.
\end{proof}

We then prove Theorem \ref{T6} with the following two lemmas.

\begin{lem}\label{lem1}
For any $p\geq 5$, $n\geq 14$, no player has a winning strategy.
\end{lem}

\begin{proof}
Suppose player $m$ has a winning strategy.\\
After player $m$'s first move, the next four players can do the following:

Player $m+1: 1+1=2$

Player $m+2: 1+1=2$

Player $m+3: 1+1=2$

Player $m+4: 2+2+2=1+5$.\\
Since $p\geq 5$, $m+1,\ m+2,\ m+3$ and $m+4$ are not congruent to $m \bmod{p}$, so player $m$ does not make any of the listed moves. These steps contradict Property $1$, thus Lemma \ref{lem1} is proved.
\end{proof}

\begin{lem}\label{lem2}
 For any $p=4$, $n\geq 16$, no player has a winning strategy.
\end{lem}

\begin{proof}
Suppose player $m$ has a winning strategy.\\
After player $m$'s first move, the next players can do the following:

Player $m+1: 1+1=2$ $($Step $1)$

Player $m+2: 1+1=2$ $($Step $2)$

Player $m+3: 1+1=2$ $($Step $3)$

Player $m:$ player $m$ can do anything $($Step $4)$

Player $m+1: 1+1=2$ $($Step $5)$

Player $m+2: 1+1=2$ $($Step $6)$

Player $m+3: 2+2+2=1+5$ $($Step $7)$.\\
If player $m$ does $2+2+2=1+5$ in Step $4$, it will violate Property $1$, a contradiction.\\
If player $m$ does anything in Step $4$ other than $2+2+2=1+5$, then Step $4$ will take away at most two $2$'s. Also note that Steps $1, 2, 3, 5$ have generated four $2$'s in total, so there will be at least two $2$'s remaining after Step $5$.\\
Therefore, the player at Step $6$ can take $1+2+2=5$ instead, and now player $m-1$ has the winning strategy.\\
Therefore, by showing that the winning strategy can be stolen, Lemma \ref{lem2} is proved.
\end{proof}

By Lemmas \ref{lem1} and \ref{lem2}, Theorem \ref{T6} is proved. \hfill $\Box$

\begin{proof}[Proof of Theorem \ref{T7}]
Suppose player $2$ has a winning strategy.\\
For any $n\geq 6$, we know that player $1$ and player $2$ both must do $1+1=2$ as their first step. We can let player $3$ also do $1+1=2$ as their first step and we can let player $1$ do $2+2=1+5$ as their second step. Therefore, if player $2$ has a winning strategy, then player $2$ must have a winning strategy for paths starting in this form:

Player $1: 1+1=2$

Player $2: 1+1=2$

Player $3: 1+1=2$

Player $1: 2+2+2=1+5$.\\
This violates Property $1$.
So by contradiction, we have proved that Theorem \ref{T7} is true for any $n\geq 6$.\\
Also, note that when $n=5$, player $3$ always has a winning strategy $($player $1$ and player $2$ both must do $1+1=2$ as their first step, so player $3$ can win the game by doing $1+2+2=5)$.\\
Thus, Theorem \ref{T7} is proved.
\end{proof}

\subsection{A Strategic Alliance}

\begin{proof}[Proof of Theorem \ref{T8}]
Suppose that the small alliance has a winning strategy. We define the first round starting from the big alliance's first move. For the first $m$ rounds, let all the players from the big alliance (consisting of $3m$ consecutive players) do $1+1=2$. \\
\noindent Case $1$: If in one of the first $m$ rounds, every player from the small alliance (consisting of $m$ consecutive players) does $2+2+2=1+5$ in this round, then after these $m$ moves, the last $m$ consecutive players of the big alliance can all do $1+2+2=5$ in the next round.

Suppose the small alliance has a winning strategy, then for any winning path, there will be a player $q$ from the small alliance who takes the last step. By using the stealing strategy mentioned previously (last $m$ consecutive in the big alliance do $1+2+2=5$ instead), player $q-m$ now becomes the player who takes the last step. Note that player $q-m$ belongs to the big alliance, so the big alliance now has the winning strategy, which leads to a contradiction. \\

\noindent Case $2$: If for each of the first $m$ rounds, at least one player in the small alliance does not do $2+2+2=1+5$, then there will be at least one $2$ generated in each round. This is because the player who does not do $2+2+2=1+5$ can only take away at most two $2$'s in that step, the small alliance can take away at most $(3m-1)$ $2$'s in each round, and the big alliance generates $3m$ $2$'s in each round. Therefore, each round can generate at least one $2$.

Thus after $m$ rounds, there will be at least $m$ $2$'s generated. In the $(m+1)$th round, the big alliance can perform $2m$ consecutive $1+1=2$ moves followed by $m$ consecutive $2+2+2=1+5$ moves. Note that in this round, the middle $m$ consecutive players of the big alliance can instead do $1+2+2=5$.

Suppose the small alliance has a winning strategy, so there is a player $q$ from the small alliance who takes the last step. By the stealing strategy mentioned above, player $q-m$ now takes the last step. Since player $q-m$ belongs to the big alliance, the big alliance now has the winning strategy, a contradiction.

Thus by Cases $1$ and $2$, Theorem \ref{T8} is proved.
\end{proof}

\subsection{A Game Without Splitting Moves}

\begin{proof}[{Proof of Theorem \ref{T9}}]
Recall the conditions of the theorem: one player, who we will henceforth call the protagonist, must be using the strategy of making, on each turn, the first move that is available of  $C_3$, $C_2$, $C_4$, $C_5$, $\ldots$, $C_{k}$, $C_1$. We first prove that no splitting move will be playable on the antagonist's turn if our protagonist is using this strategy. We do this by induction on the size of the index $n$ where the splitting move $S_n$ is being made.

Our base case is then $n=2$. Here we induct on the protagonist's turns. After the protagonist's first turn, there can be no splitting moves, since it can be at most the second turn and there can be at most two $2$'s. Now we assume the inductive hypothesis: after $i$ turns for the protagonist, we have at most two $2$'s. As long as we have not played enough of the game to make $S_3$ or $S_4$ \footnote{This must true the first time $C_2$ is available, as we must make either $C_2$ or $S_2$ to get $a_3$. We will go on to prove that without larger splitting moves, $S_2$ is impossible. Then we will prove that without $S_2$, larger splitting moves are impossible. Thus, neither $S_2$ nor larger splitting moves will be possible.}, the antagonist can only play $C_1$ to increase the number of $2$'s. If we now have two or three 2's, since there is a $1$ available the protagonist will play either $C_2$ or $C_3$. After either of these moves, there are once again two or fewer $2$'s remaining. If not, the protagonist can make any move and there will still be two or fewer 2's remaining. Thus, by induction, $S_2$ can never be played.

Now, as long as $S_4$ and $S_5$ are not played, we will prove that neither player will play $S_3$. Given this, the only way to increase the number of $5$'s is to play $C_2$. Thus to get to four $5$'s, we must first have two $2$'s and three $5$'s, and either it is the protagonist's turn at this point and they will play $C_3$ or on the previous turn there must have been at least one $2$ and three $5$'s, from which point the protagonist would play $C_3$.

For $n>3$, the inductive hypothesis states that no one will play $S_j$ for $j<n$: we will prove that $S_n$ will not be played as long as $S_m$ cannot be played for $m>n$. Thus the only way to increase the value of $a_n$ is to play $C_{n-1}$. To get to $a_n^{n+1}$, we must first have $a_n^{n}\wedge a_{n-1}^{n-1}\wedge a_{n-2}$. Thus $C_{n}$ must be available before $S_n$. Let us consider when $C_n$ first becomes available. We must either make $C_{n-1}$ or $C_{n-2}$ to get to this point: for $C_{n-1}$ to be made there must have been $a_{n-1}^{n}$ and therefore $C_{n-2}$ must have been made once in the two preceding turns, otherwise $S_{n-1}$ would have been available to the antagonist. Thus for $S_{n-2}$ to never have been available to the antagonist, there must be at most one $a_{n-2}$ once $C_n$ becomes available.

Now for $S_n$ to be playable, the players must first make more of $a_{n-2}$. This requires playing $C_{n-3}$, which then means we have at most one $a_{n-3}$, since otherwise the antagonist must have had the opportunity to play $S_{n-3}$. Then, to increase the value of $a_{n-3}$ we must play $C_{n-4}$ and have at most one $a_{n-4}$. It continues, descending, that for successive $2< p <n-4$ directly after someone has played $C_{p}$ there is at most one $a_p$ and at most two $a_{p+1}$'s, until finally, $C_2$ is played and we have at most one $2$ and at most two $5$'s. If it's the protagonist's turn, they will play $C_n$, eliminating the possibility of $S_n$. If it's the antagonist's turn, only $C_1$ and $C_n$ can be made, so $C_1$ is their only useful move. Now the protagonist makes $C_2$ if it is available and $C_n$ if it is not. The antagonist, not wanting to play $C_n$, will play $C_1$ again. If the protagonist can play $C_3$, they do so, and if not, they can play $C_n$. The antagonist will play $C_1$, and the protagonist will play $C_n$. Thus, the protagonist has successfully prevented $S_n$ from being played. So by induction, the protagonist can force the game to progress without a splitting move.
\end{proof}

\begin{conj}
Using the following strategy, either player can force the game to be played to completion without a splitting move when the game is played on $n=a_i$, for $a_i$ a term in the sequence.
\end{conj}

The fact that each term in the sequence conforms to the equality $a_i = a_{i-1}^{i-1} \wedge a_{i-3}^{i-3} \wedge \ldots$ suggests that this is true, as we ought to have $1$'s until we get to a game state that looks like the right side of the equation, after which point only combining moves will be available. However, we have not yet proven that there is no other way for the game to progress.

\section{Future Work}

There are several unanswered questions that may interest other researchers.
\begin{itemize}
\item Can we determine the distribution of gaps of any size?
\item Can we prove Gaussianity?
\item Are there other nonlinear recurrences that have unique decompositions? Do we have similar results for the distribution of gaps and number of summands?
\item For the generalized Zeckendorf game, through simulations we were able to tighten the bound on a Combine Only game to about $0.6601\:n$. What does $MC(a_i)$ converge to as $i\rightarrow\infty$?
\item While we have results about winning strategies or the lack of them pertaining to game with $p\geq3$, showing the existence of winning strategies for the two-player game remains unsolved. One possible strategy we have been exploring is the Combine Only game as a winning strategy.
\item Can an upper bound be found on the number of moves in a general game (without specific restrictions on moves)?
\item What other positive nonlinear recurrence sequences can the game be extended to?
\end{itemize}

\ \\


\begin{thebibliography}{BBGILMTW1}

    \bibitem[BEFM1]{BEFM1} P. Baird-Smith, A. Epstein, K. Flint and S. J. Miller, \textit{The Generalized Zeckendorf Game}, Proceedings of the 18th International Conference on Fibonacci Numbers and Their Applications, Fibonacci Quarterly \textbf{57} (2019), no. 5, 1--14.

    \bibitem[BEFM2]{BEFM2} P. Baird-Smith, A. Epstein, K. Flint and S. J. Miller, \textit{The Zeckendorf Game}, Combinatorial and Additive Number Theory III, CANT, New York, USA, 2017 and 2018, Springer Proceedings in Mathematics \& Statistics \textbf{297} (2020), 25--38.

    \bibitem[BBGILMT]{BBGILMT} Olivia Beckwith, Amanda Bower, Louis Gaudet, Rachel Insoft, Shiyu Li, Steven J. Miller and Philip Tosteson, \emph{The Average Gap Distribution for Generalized Zeckendorf Decompositions}, Fibonacci Quarterly \textbf{51} (2012), no. 1, 13--27 .

    \bibitem[BDEMMTTW]{Bes}
    A. Best, P. Dynes, X. Edelsbrunner, B. McDonald, S. Miller, K. Tor, C. Turnage-Butterbaugh, M. Weinstein, \emph{Gaussian Behavior of the Number of Summands in Zeckendorf Decompositions in Small Intervals}, Fibonacci Quarterly, \textbf{52} (2014), no. 5, 47--53.

    \bibitem[BILMT]{Bow} 
    A. Bower, R. Insoft, S. Li, S. J. Miller and P. Tosteson, \emph{The Distribution of Gaps between Summands in Generalized Zeckendorf Decompositions} (and an appendix on \emph{Extensions to Initial Segments} with Iddo Ben-Ari), Journal of Combinatorial Theory, Series A \textbf{135} (2015), 130--160.

    \bibitem[Br]{Br}
    J. L. Brown, Jr., \emph{Zeckendorf's Theorem and Some Applications}, The Fibonacci Quarterly, \textbf{2} (1964), no. 3, 163--168.

    \bibitem[Day]{Day}
    D. E. Daykin,  \emph{Representation of Natural Numbers as Sums of Generalized Fibonacci Numbers}, J. London Mathematical Society \textbf{35} (1960), 143--160.

    \bibitem[DDKMMV]{Dem}
    P. Demontigny, T. Do, A. Kulkarni, S. J. Miller, D. Moon and U. Varma, \emph{Generalizing Zeckendorf's Theorem to $f$-decompositions}, Journal of Number Theory \textbf{141} (2014), 136--158.

    \bibitem[GTNP]{GTNP}
    P. J. Grabner, R. F. Tichy, I. Nemes, and A. Peth\"o, \emph{Generalized Zeckendorf expansions},
    Appl. Math. Lett. \textbf{7} (1994), no. 2, 25--28.

    \bibitem[FGNPT]{FGNPT}
    P. Filipponi, P. J. Grabner, I. Nemes, A. Peth\"o, and R. F. Tichy, \emph{Corrigendum
    to: ``Generalized Zeckendorf expansions''}, Appl. Math. Lett.
    \textbf{7} (1994), no. 6, 25--26.

    \bibitem[Fr]{Fr}
    \newblock A. S. Fraenkel, \emph{Systems of Numeration}, Amer. Math. Monthly \textbf{92} (1985), no. 2, 105--114.

    \bibitem[Ha]{Ha} N. Hamlin,  \emph{Representing Positive Integers as a Sum of Linear Recurrence Sequences}, Abstracts of Talks, Fourteenth International Conference on Fibonacci Numbers and Their Applications (2010), 2--3.

    \bibitem[HW]{HW}
    N. Hamlin and W. A. Webb, \emph{Representing positive integers as a sum of linear recurrence sequences}, Fibonacci Quarterly \textbf{50} (2012), no. 2, 99--105.

    \bibitem[Ho]{Ho} V. E. Hoggatt,  \emph{Generalized Zeckendorf theorem}, Fibonacci Quarterly \textbf{10} (1972), no. 1 (special issue on representations), 89--93.

    \bibitem[Ke]{Ke} T. J. Keller,  \emph{Generalizations of Zeckendorf's theorem}, Fibonacci Quarterly \textbf{10} (1972), no. 1 (special issue on representations), 95--102.

    \bibitem[KKMW]{KKMW} M. Kolo$\breve{{\rm g}}$lu, G. Kopp, S. J. Miller and Y. Wang,  \emph{On the number of Summands in Zeckendorf Decompositions}, Fibonacci Quarterly \textbf{49} (2011), no. 2, 116--130.

    \bibitem[Lek]{Lek}
    C. G. Lekkerkerker,  \emph{Voorstelling van natuurlyke getallen door een som van getallen van Fibonacci}, Simon Stevin \textbf{29} (1951-1952), 190--195.

    \bibitem[LM1]{LM1}
    R. Li and S. J. Miller, \emph{A Collection of Central Limit Type results in Generalized Zeckendorf Decompositions}, the 17th International Fibonacci Conference, Fibonacci Quarterly \textbf{55} (2017), no. 5, 105--114.

    \bibitem[LM2]{LM2} 
    R. Li and S. J. Miller, \emph{Central Limit Theorems for Gaps of Generalized Zeckendorf Decompositions}, Fibonacci Quarterly \textbf{57} (2019), no. 3, 213--230.

    \bibitem[LLMMSXZ]{LLMMSXZ} R. Li, X. Li, S. J. Miller, C. Mizgerd, C. Sun, D. Xia, Z. Zhou (2020) \textit{Deterministic Zeckendorf Games}.

    \bibitem[Li]{Li} J.W. Lindeberg, \emph{Eine neue Herleitung des Exponentialgesetzes in der Wahrscheinlichkeitsrechnung}, Mathematische Zeitschrift \textbf{15} (1922), 211--225.

    \bibitem[MW1]{MW1} S. Miller, Y. Wang, \emph{From Fibonacci Numbers to Central Limit Type Theorems}, Journal of Combinatorial Theory, Series A 119 (2012), no. 7, 1398--1413.

	\bibitem[MW2]{MW2} S. Miller, Y. Wang, \emph{Gaussian Behavior in Generalized Zeckendorf Decompositions}, Combinatorial and Additive Number Theory, CANT 2011 and 2012 (Melvyn B. Nathanson, editor), Springer Proceedings in Mathematics \& Statistics (2014), 159--173.

    \bibitem[Ste1]{Ste1}
    W. Steiner, \emph{Parry expansions of polynomial sequences}, Integers \textbf{2} (2002), Paper A14.

    \bibitem[Ste2]{Ste2}
    W. Steiner, \emph{The Joint Distribution of Greedy and Lazy Fibonacci Expansions}, Fibonacci Quarterly \textbf{43} (2005), 60--69.

    \bibitem[Za]{Za} S. L. Zabell, \emph{Alan Turing and the Central Limit Theorem}, Amer. Math. Monthly \textbf{102} (1995), 483--494.

    \bibitem[Ze]{Ze}
    E. Zeckendorf, \emph{Repr\'esentation des nombres naturels par une somme de nombres de Fibonacci ou de nombres de Lucas}, Bull. Soc. Roy. Sci. Li\`ege \textbf{41} (1972), 179--182.
\end{thebibliography}
\end{document}